\documentclass[11pt]{amsart}

\usepackage{latexsym}
\usepackage{amsmath}
\usepackage{amsfonts}
\usepackage{amssymb}
\usepackage{tikz,float}
\usepackage{cite}
\usepackage{enumitem}   

\newtheorem{theorem}{Theorem}[section]
\newtheorem{lemma}[theorem]{Lemma}
\newtheorem{corollary}[theorem]{Corollary}
\newtheorem{proposition}[theorem]{Proposition}

\newtheorem{sublemma}{}[theorem]

\theoremstyle{definition}

\theoremstyle{remark}

\numberwithin{equation}{section}

\DeclareMathOperator{\cl}{cl}
\DeclareMathOperator{\si}{si}

\begin{document}

\title[Matroids Arising From Nested Sequences of Flats]{Matroids Arising From Nested Sequences of Flats In Projective And Affine Geometries}

\author{Matthew Mizell}
\address{Mathematics Department \\
Louisiana State University \\ 
Baton Rouge, Louisiana}
\email{mmizel4@lsu.edu}

\author{James Oxley}
\address{Mathematics Department \\
Louisiana State University \\ Baton Rouge, Louisiana}
\email{oxley@math.lsu.edu}

\date{\today}

\begin{abstract}
    Targets are matroids that arise from a nested sequence of flats in a projective geometry. This class of matroids was introduced by Nelson and Nomoto, who found the forbidden induced restrictions for binary targets. This paper generalizes their result to targets arising from projective geometries over $GF(q)$. We also consider targets arising from nested sequences of affine flats and determine the forbidden induced restrictions for affine targets. 
\end{abstract}
\maketitle
\section{Introduction}
Throughout this paper, we follow the notation and terminology of [\ref{James}]. All matroids considered here are simple. This means, for example, that when we contract an element, we always simplify the result. An \textit{induced restriction} of a matroid $M$ is a restriction of $M$ to one of its flats.

Let $M$ be a rank-$r$ projective or affine geometry represented over $GF(q)$. We call $(F_0,F_1,\dots,F_k)$ a \textit{nested sequence of projective flats} or a \textit{nested sequence of affine flats} if $\emptyset = F_0 \subseteq F_1 \subseteq \dots \subseteq F_{k-1} \subseteq F_k =  E(M)$ and each $F_i$ is a, possibly empty, flat of $M$. Let $(G,R)$ be a partition of $E(M)$ into, possibly empty, subsets $G$ and $R$. We call the elements in $G$ \textit{green}; those in $R$ are \textit{red}. A subset $X$ of $E(M)$ is \textit{monochromatic} if $X\subseteq G$ or $X\subseteq R$. For a subset $X$ of $E(PG(r-1,q))$, we call $PG(r-1,q)|X$ a \textit{projective target}, or a \textit{target}, if there is a nested sequence $(F_0,F_1,\dots,F_k)$ of projective flats such that $X$ is the union of all sets $F_{i+1}-F_i$ for $i$ even. It is straightforward to check that $PG(r-1,q)|G$ is a target if and only if $PG(r-1,q)|R$ is a target. Because $GF(q)$-representable matroids are not necessarily uniquely $GF(q)$-representable, we have defined targets in terms of $2$-colorings of $PG(r-1,q)$. When $X \subseteq E(AG(r-1,q))$, we call $AG(r-1,q)|X$ an \textit{affine target} if there is a nested sequence $(F_0,F_1,\dots,F_k)$ of affine flats such that $X$ is the union of all sets $F_{i+1} - F_i$ for $i$ even. For affine targets in $AG(r-1,q)$, we follow the same convention of defining targets in terms of $2$-colorings. 

Consider an analogous construction for graphs, that is, take a sequence $(K_0,K_1,\dots,K_n)$ of complete graphs where $K_{i+1}$ has $K_i$ as a subgraph for each $i$ in $\{1,2,\dots,n-1\}$. Moreover, for each such $i$, color the vertex $v$ of $V(K_{i+1}) - V(K_i)$ either green or red and color all the edges of $E(K_{i+1}) - E(K_i)$ the same color as $v$. This process is equivalent to repeatedly adding green or red dominating vertices, that is, adding a green or red vertex $v$ to a graph $G$ that is adjacent to every vertex $u$ in $V(G - v)$. Consider the subgraph $H$ of $K_n$ whose vertex set is $V(K_n)$ and whose edge set is the set of green edges. Observe that, in the construction of $H$, when a red dominating vertex is added, it is an isolated vertex of the graph that has been constructed so far. Therefore, to construct $H$, at each step, we are adding a green dominating vertex or a red isolated vertex. Chvátal and Hammer [\ref{thresholdgraphs}] showed that the class of graphs that arises from repeatedly adding dominating vertices and isolated vertices coincides with the class of threshold graphs. This is the class of graphs that has no induced subgraph that is isomorphic to $C_4,2K_2$, or $P_4$, that is a $4$-cycle, two non-adjacent edges, or a $4$-vertex path.

Nelson and Nomoto [\ref{NELSON}] introduced binary targets and characterized them as follows.
\begin{theorem}\label{binarytarget}
Let $(G,R)$ be a $2$-coloring of $PG(r-1,2)$. Then $PG(r-1,2)|G$ is a target if and only if $it$ does not contain $U_{3,3}$ or $U_{2,3}\oplus U_{1,1}$ as an induced restriction.
\end{theorem}

Nelson and Nomoto [\ref{NELSON}] call $U_{3,3}$ the \textit{claw}, while they call $U_{2,3} \oplus U_{1,1}$, the complement of $U_{3,3}$ in $F_7$, the \textit{anti-claw}. They derive Theorem~\ref{binarytarget} as a consequence of a structural description of claw-free binary matroids. In Section~\ref{projmainthm}, we give a proof of that theorem that does not rely on this structural description. Then, for all $q \geq 3$, we  
characterize targets represented over $GF(q)$ in terms of forbidden induced restrictions by proving the next result.
\begin{theorem}\label{qarytarget}
    For a prime power $q$ exceeding two, let $(G,R)$ be a $2$-coloring of $PG(r-1,q)$. Then $PG(r-1,q)|G$ is a target if and only if it does not contain any of $U_{2,2},U_{2,3},\dots, U_{2,q-2},\text{or } U_{2,q-1}$ as an induced restriction.
    
\end{theorem}

In Section 2, we prove some useful properties of targets. In particular, we show that targets are closed under contractions. A simple matroid $N$ is an \textit{induced minor} of a simple matroid $M$ if $N$ can be obtained from $M$ by a sequence of contractions and induced restrictions. We observe that Theorem~\ref{qarytarget} can also be viewed as characterizing targets in terms of forbidden induced minors. Our other main theorems, which are proved in Section \ref{affinetargets}, characterize affine targets in terms of forbidden induced restrictions. There are three cases, depending on the value of $q$.

\begin{theorem}\label{affinetargetq2}
Let $(G,R)$ be a $2$-coloring of $AG(r-1,2)$. Then $AG(r-1,2)|G$ is an affine target if and only if it does not contain $U_{4,4}$ as an induced restriction.
\end{theorem}
 The matroids $\mathcal{W}^3$ and $P(U_{2,3}, U_{2,3})$ that appear in the next theorem are the rank-$3$ whirl and the parallel connection of two copies of $U_{2,3}$.
\begin{theorem}\label{affinetargetq3}
    Let $(G,R)$ be a $2$-coloring of $AG(r-1,3)$. Then $AG(r-1,3)|G$ is an affine target if and only if it does not contain any of $U_{3,3}, U_{3,4},\linebreak U_{2,3}\oplus U_{1,1}, U_{2,3} \oplus_2 U_{2,4}, P(U_{2,3},U_{2,3}),$ or $\mathcal{W}^3$ as an induced restriction.
\end{theorem}

\begin{theorem}\label{affinetargetsq4}
    Let $(G,R)$ be a $2$-coloring of $AG(r-1,q)$, for $q \geq 4$. Then $AG(r-1,q)|G$ is an affine target if and only if it does not contain any of $U_{2,2},U_{2,3},\dots, U_{2,q-3}, \text{or } U_{2,q-2}$ as an induced restriction.
\end{theorem}

\section{Preliminary Results}
Throughout the paper, we will refer to flats and hyperplanes of $PG(r-1,q)$ as \textit{projective flats} and \textit{projective hyperplanes}, respectively. Let $M$ be a restriction of $PG(r-1,q)$. For a subset $X$ of $E(M)$, its \textit{projective closure}, $\cl_P(X)$, is the closure of $X$ in the matroid $PG(r-1,q)$. We first show that if $PG(r-1,q)|G$ is a target, then the matroid $PG(r-1,q)|G$ is uniquely determined by the sequence $(r_0,r_1,\dots, r_k)$ of ranks of the nested sequence $(F_0,F_1,\dots, F_k)$ of projective flats. Note that we shall often write $G$ and $R$ for the matroids $PG(r-1,q)|G$ and $PG(r-1,q)|R$, respectively. This means that we will be using $G$ and $R$ to denote both matroids and the ground sets of those matroids.

\begin{proposition}\label{projwelldefined}
Let $(E_0,E_1,\dots, E_k)$ and $(F_0,F_1, \dots, F_k)$ be nested sequences of flats in $PG(r-1,q)$ such that $r(E_i) = r(F_i)$ for all $i$ in $\{0,1,\dots k\}$. Let $G_E$ and $G_F$ be the union, respectively, of all $E_{i+1}-E_i$ and of all $F_{i+1}-F_i$ for the even numbers $i$ in $\{0,1,\dots,k\}$. Then $PG(r-1,q)|G_E \cong PG(r-1,q)|G_F$.
\end{proposition}
\begin{proof}
    Let $h$ be the smallest $i$ such that $r(E_i) > 0$. Let $\{b_{h,1},b_{h,2},\dots,b_{h,m_h}\}$ and $\{d_{h,1},d_{h,2},\dots,d_{h,m_h}\}$ be bases $B_h$ and $D_h$ of $PG(r-1,q)|E_h$ and $PG(r-1,q)|F_h$, respectively. Let $B_0 = B_1 = \dots = B_{h-1} = \emptyset$ and $D_0 = D_1 = \dots = D_{h-1} = \emptyset$. For $j \geq h$, assume that $B_0, B_1,\dots, B_j$ and $D_0,D_1,\dots,D_j$ have been defined. Let $B_{j+1}$ and $D_{j+1}$ be bases of $E_{j+1}$ and $F_{j+1}$, respectively, such that $B_{j} \subseteq B_{j+1}$ and $D_j~\subseteq~D_{j+1}$. Let $B_{j+1} - B_j = \{b_{j+1,1},b_{j+1,2},\dots,b_{j+1,m_{j+1}}\}$ and $D_{j+1} - D_j = \{d_{j+1,1}, d_{j+1,2},\dots,d_{j+1,m_{j+1}}\}$. Define the automorphism $\phi$ on $PG(r-1,q)$ by $\phi(b_{s,t})=d_{s,t}$ for all $s$ and $t$ such that $s \geq h$. Then $\phi(E_i) = F_i$ for all $i$, so $\phi(E_{i+1} - E_i) = \phi(E_{i+1}) - \phi(E_i) = F_{i+1} - F_i$, for all $i$. Therefore, $PG(r-1,q)|G_E \cong PG(r-1,q)|G_F$.
\end{proof}
The last result means that we can refer to a simple $GF(q)$-representable matroid $M$ as being a target exactly when some, and hence all, of the $GF(q)$-representations of $M$ are targets. Note that in a nested sequence $(F_0,F_1,\dots,F_k)$ of flats defining a target, it is convenient to allow equality of the flats. A nested sequence $(F_0,F_1,\dots,F_k)$ of flats is the \textit{canonical nested sequence} defining a projective or affine target if $F_0 = \emptyset$, and $F_1,F_2,\dots,F_{k-1}$, and $F_k$ are distinct. Observe that allowing $F_1$ to be empty accommodates the requirement that the target is the union of all sets $F_{i+1} - F_i$ for $i$ even.

Lemma~$2.15$ of Nelson and Nomoto [\ref{NELSON}] proved that binary targets are closed under induced restriction. Using the same proof, their result can be extended to targets represented over $GF(q)$.
\begin{lemma} \label{targetinudcedres}
    The class of targets over $GF(q)$ is closed under induced restrictions.
\end{lemma}
 
 \begin{lemma} \label{fullrank} 
Let $(G,R)$ be a $2$-coloring of $PG(r-1,q)$. Assume that $G$ is a target and $F$ is a projective flat of $PG(r-1,q)$. Then exactly one of $G \cap F$ and $R \cap F$ has rank $r(F)$.
\end{lemma}
\begin{proof}
By Lemma~\ref{targetinudcedres}, $PG(r-1,q)|(G\cap F)$ is a target corresponding to a nested sequence $(F_0',F_1',\dots, F'_{k-1}, F)$ of projective flats. By, for example, [\ref{SINGH}, Lemma~2.1], $r(G\cap F)$ or $r(R\cap F)$ is $r(F)$. Either $G\cap F$ or $R\cap F$ is contained in some proper projective flat of $F$. Therefore, either  $r(G\cap F) < r(F)$ or $r(R \cap F) < r(F)$.
\end{proof}
 We refer to the rank of the set of green elements in a projective flat $F$ as the \textit{green rank} of $F$. If $F$ has green rank $r(F)$, we say that $F$ is a \textit{green flat}. Furthermore, if a projective hyperplane has green rank $r-1$, then it is a \textit{green hyperplane}. Red rank, red flats, and red hyperplanes are defined analogously. From the last lemma, it follows that a projective flat can either be a green flat or a red flat, but not both.
 
 We now show that every contraction of a target is a target. Consider contracting a green element $e$ in $ M$. If a parallel class in the contraction contains at least one green point, then, after the simplification, the resulting point will be green. If there are only red points in the parallel class, then, after the simplification, the resulting point is red.


\begin{proposition} \label{targetcontraction}
The class of targets over $GF(q)$ is closed under contractions. 
\end{proposition}
\begin{proof}    
Let $(G,R)$ be a $2$-coloring of $PG(r-1,q)$. Assume that $G$ is a target. Then there is a canonical nested sequence $(F_0, F_1, \dots, F_k)$ of projective flats such that $G$ is the union of all sets $F_{i+1} - F_i$ for $i$ even. Let $e$ be an element of $ F_m - F_{m-1}$ where $F_m$ is a green flat. Then the elements of $F_m - F_{m-1}$ are green. Suppose $x$ is a red point in $F_m$. Then $x \in F_{m-1}$. If $y \in \cl_P(\{e,x\})$, then $y \not\in F_{m-1}$, otherwise the circuit $\{e,x,y\}$ gives the contradiction that $e$ is an element of  $F_{m-1}$. Since $\{e,x\} \subseteq F_m$, we must have that $y$ is in $ F_m$, so $y$ is in $ F_m - F_{m-1}$. Hence $y$ is green. We deduce that, in the contraction of $e$, every element of $F_m - e$ is green.

Now assume $F_j$ is a red flat containing $F_m$. Then $F_j - F_{j-1} \subseteq R$. Consider a  point $z \text{ in } F_j - F_{j-1}$. Using a symmetric argument to that given above, we deduce that $e$ is the only point of $\cl_P(\{e,z\})$ not in $F_j-F_{j-1}$. Therefore, the points in $(F_j - F_{j-1}) - e$ are red. Clearly, if $F_k$ is a green flat containing $F_m$, then the points in $(F_k - F_{k-1}) - e$ are green. Thus, in $\si(PG(r-1,q) / e)$, we have $(\si(F_m - e), \si(F_{m+1} - e),\dots,\si(F_k - e))$ as a nested sequence of projective flats. Writing this new nested sequence of projective flats in $PG(r-2,q)$ as $(F_m',F_{m+1}',\dots, F_k')$, we see that $F_m'$ is entirely green and, for each $i\geq 1$, the set $F_{m+i}' - F_{m+i-1}'$ is entirely red if $i$ is odd and is entirely green if $i$ is even. Hence $\si(G/e)$ is a target.\end{proof}

Combining Lemma~\ref{targetinudcedres} and Proposition \ref{targetcontraction}, we get the following.
\begin{corollary}
    The class of targets over $GF(q)$ is closed under induced minors.
\end{corollary}\label{inducedminor}

\begin{lemma}
    Let $(G,R)$ be a $2$-coloring of $PG(r-1,q)$. If $G$ is a target, then $G$ and $R$ are connected unless $q = 2$ and $G$ or $R$ is $U_{2,2}$.
\end{lemma}

\begin{proof}
Assume that the exceptional case does not arise and that $r(G) \geq r(R)$. If $G = PG(r-1,q)$, then the result holds. Assume $G$ is not the whole projective geometry. Then $G$ contains $AG(r-1,q)$, so $G$ is connected. Similarly, $R$ will also have an affine geometry as a restriction. Thus $R$ is certainly connected when $r(R) = r(G)$. Assume $r(R) < r(G)$. Take a projective flat $F$ that has $R$ as a spanning restriction. Then $r(R) \geq r(G\cap F)$ so, as above, we deduce that $R$ is connected. 
\end{proof}

If $(G,R)$ is a $2$-coloring of $PG(r-1,q)$, then $G$ is a \textit{minimal non-target} if $G$ is not a  target but every proper induced restriction of $G$ is a target. Clearly, if $G$ is a minimal non-target, then $R$ is not a target. But if $r(R) > r(G)$, then $R$ is not a minimal non-target.

\begin{lemma} \label{rankcomp}
Let $(G,R)$ be a $2$-coloring of $PG(r-1,q)$. Suppose $PG(r-1,q)|G$ is a minimal non-target of rank $r$. Then $r(R) = r$.
\end{lemma}

\begin{proof}
Assume $r(R) < r$. Then there is a hyperplane $H$ containing $R$. Since $PG(r-1,q)|(G\cap H)$ is a target, $R$ is a target. However, this implies that $G$ is a  target, a contradiction. Therefore $r(R) = r$.
\end{proof}

\section{Forbidden Induced Restrictions of Target Matroids}\label{projmainthm}

This section contains a common proof of Theorems \ref{binarytarget} and \ref{qarytarget}. This proof closely follows the proof of Theorem $4.7$ of Singh and Oxley [\ref{SINGH}].


\begin{proof}[Proof of Theorems \ref{binarytarget} and \ref{qarytarget}]
Assume that $G$ is a target. First, suppose $q = 2$. If there is a projective flat $F$ such that $PG(r-1,2)|(G\cap F) \cong U_{3,3}$, then $PG(r-1,2)|(R \cap F) \cong U_{2,3} \oplus U_{1,1}$. Since $PG(r-1,2)|(G\cap F)$ is a target, this contradicts Lemma~\ref{fullrank}, as $r(G \cap F) = r(R \cap F)$. Now assume $q \geq 3$. If there is a projective flat $F$ such that $PG(r-1,q)|(G\cap F)$ is any of $U_{2,2}, U_{2,3}, \dots, U_{2,q-2},$ or $U_{2,q-1}$, then, letting $F' = \cl_P(G\cap F)$, we have $r(G\cap F') = r(R\cap F')$, a contradiction to Lemma~$\ref{fullrank}$.

Let $(G,R)$ be a $2$-coloring of $PG(r-1,q)$. Suppose that $G$ is a rank-$r$ minimal non-target. In addition, when $q = 2$, assume that $G$ does not have $U_{3,3}$ or $U_{2,3}\oplus U_{1,1}$ as an induced restriction; and when $q \geq 3$, assume instead that $G$ does not have $U_{2,2},U_{2,3}, \dots,U_{2,q-2},$ or $U_{2,q-1}$ as an induced restriction. Then, by Lemma~\ref{rankcomp},  $r(R) = r$. Clearly, $r\geq 4$ when $q = 2$, and  $r\geq 3$ when $q \geq 3$.
\setcounter{theorem}{1}
\begin{sublemma} \label{binaryflat}
When $q\geq 2$, each green hyperplane $H$ contains at most one red rank-$(r-2)$ flat.
\end{sublemma}
Assume that $H$ contains at least two red flats, $F_1$ and $F_2$, of rank $r-2$. Then, all the elements of $F_1 - F_2$ are red. Adding an element $z$ of $F_2 - F_1$ to $F_1 - F_2$ gives a subset of $H$ whose red rank is $r-1$. This is a contradiction as $H$ is a green hyperplane.  Thus \ref{binaryflat} holds.

Consider a rank-$(r-2)$ projective flat $F$. Then $F$ is contained in exactly $q+1$ projective hyperplanes. Assume that $r(R\cap F) = r(F)$. We make the following observations.

\begin{sublemma}\label{binaryhyperplane}
When $q=2$, at most two of $H_1,H_2$ and $H_3$ are green.
\end{sublemma}
Assume that all three hyperplanes are green. Then, all the elements of each of $H_1 - F, H_2 - F$ and $H_3 - F$ are monochromatic green, so $r(R) = r(F) = r-2$, a contradiction to Lemma~\ref{rankcomp}. Thus \ref{binaryhyperplane} holds.
\begin{sublemma}\label{2redhyperplanes}
    When $q\geq 3$, there are at least two red hyperplanes containing $F$.
\end{sublemma} 
 As $r(R) = r$, there is at least one red hyperplane containing $F$. Now, assume there is exactly one red hyperplane containing $F$. Then, as $r(R) = r$, there is some red point in a green hyperplane $H_G$ that contains $F$. Therefore $r(G\cap H_G) = r(R\cap H_G)$, contradicting Lemma~\ref{fullrank}. Thus \ref{2redhyperplanes} holds.
\begin{sublemma}\label{BadLineHP}
    When $q\geq 3$, there is at most one green hyperplane containing $F$.
\end{sublemma} 
Let $H_{G_1}$ and $H_{G_2}$ be distinct green hyperplanes containing $F$ and let $H_{R_1}$ and $H_{R_2}$ be distinct red hyperplanes containing $F$. If there is a red point in $H_{G_1} - F$, then $r(G\cap H_{G_1}) = r(R\cap H_{G_1})$, a contradiction. Hence, there are no red points in $H_{G_{1}} - F$ or in $H_{G_2}- F$. Consider red points $x \text{ in } H_{R_1} - F$ and $y \text{ in }  H_{R_2} - F$. The line $\cl_P(\{x,y\})$ intersects each of $H_{G_1}$ and $H_{G_2}$ once at some point not in $F$. Therefore, this line will have at least two red and two green points, a contradiction. Thus \ref{BadLineHP} holds.

Let $G_2$ and $R_2$ be the sets of green and red projective flats of $PG(r-1,q)$ of rank $r-2$, and let $G_1$ and $R_1$ be the sets of green and red projective hyperplanes of $PG(r-1,q)$. We now construct a bipartite graph $B$ with vertex sets $G_2 \cup R_2$ and $G_1 \cup R_1$. A vertex $x \text{ in } G_2\cup R_2$ is adjacent to a vertex $y \text{ in } G_1\cup R_1$ if the flat associated to $x$ is contained in the hyperplane associated to $y$. We count the number of cross edges, $G_1R_2$-edges or $R_1G_2$-edges, of $B$.  By \ref{binaryflat}, no flat in $G_1$ contains two or more flats in $R_2$, so it follows, using symmetry, that the total number of cross edges is at most $|G_1| + |R_1|$. Consider a pair $(H_G,H_R)$, where $H_G \in G_1$ and $H_R \in R_1$. The total number of these pairs is $|G_1||R_1|$. Say $H_G \cap H_R$ is a red flat $F_R$. Then the edge of $B$ from $H_G$ to $F_R$ is a cross edge. When $q = 2$, by \ref{binaryhyperplane}, there is at most one other red hyperplane $H'_{R}$ such that $H'_{R} \cap H_G = F_R$. Therefore, when $q = 2$, the number of cross edges is at least $\frac{1}{2}|G_1||R_1|$. If $q \geq 3$, then, by \ref{BadLineHP}, each cross edge corresponds to exactly $q$ such  pairs, so the number of cross edges is at least $\frac{1}{q}|G_1||R_1|$. Hence, for all $q$, the number of cross edges is at least $\frac{1}{q}|G_1||R_1|$. Thus, 
\begin{align}
    \frac{1}{q}|G_1||R_1| \leq |G_1| + |R_1|. \tag{1}\label{eq1}
\end{align}
We may suppose $|G_1| \leq |R_1|$. Then $\frac{1}{q}|G_1| \leq \frac{|G_1|}{|R_1|} + 1$, so 
\begin{align}
    |G_1| \leq 2q.\tag{2} \label{eq2}
\end{align}

Assume $q = 2$. Then $|G_1| \leq 4$. Now take a basis $B_G$ of $G$. As $r(G) \geq 4$, each $(r(G) -1)$-element subset of $B_G$ spans a green hyperplane. Hence $|G_1| \geq 4$, so $|G_1| = 4$. Then, by (\ref{eq1}), we have $\frac{1}{2}(4)|R_1| \leq 4 + |R_1|$, so $|R_1| \leq 4$. Therefore, $|R_1| = 4$ and $PG(r-1,2)$ has exactly eight hyperplanes. This is a contradiction, as $PG(r-1,2)$ has $2^r-1$ hyperplanes.

Now assume $q \geq 3$. Take a green hyperplane $H$. As $r(G) = r(R) =r$, there is some green point $z $ not in $ H$. Now, $PG(r-1,q)|(G\cap H)$ is a target having, say $(F_0,F_1,\dots,F_{k-1},H)$, as its corresponding canonical nested sequence of projective flats. Let $X$ be a projective flat of rank $r-2$ that is contained in $H$ and contains $F_{k-1}$. Then $PG(r-1,q)|(H - X) \cong AG(r-2,q)$ and all the elements of $H-X$ are green. In $H-X$, there are $\frac{q(q^{r-2}-1)}{q-1}$ green rank-$(r-2)$ affine flats. Let $Z$ be  one of these affine flats. Then $\cl_P(Z\cup z)$ will be a green projective hyperplane. This implies that $|G_1| \geq \frac{q(q^{r-2}-1)}{q-1}$, so, by (\ref{eq2}),
\begin{align}
2q \geq& \frac{q(q^{r-2} - 1)}{q -1}. \nonumber\end{align}
Thus $2q-2 \geq q^{r-2} - 1$, so
\begin{align}2q \geq& q^{r-2} + 1. \nonumber
\end{align}
Observe that, for $r \geq 4$, as $q \geq 3$, the last inequality does not hold. Thus $r(G)\leq 3$.

Assume $r(G) = 3$. Suppose there is a green line $L$ with a red point $z$ on it. Because there is no line having at least two green and at least two red points, $z$ is the only red point on $L$. As $r(G) = r(R) = 3$, there are red points $u$ and $v$ that are not on $L$ such that $r(\{u,v,z\}) = 3$. Then $\cl_P(\{u,v\})$ is a red line $L_1$ that meets $L$ at green a point $p$. Moreover, there is a green point $x$ that is not on $L$ or $L_1$. Consider the line $L_2 = \cl_P(\{x,z\})$. This line intersects $L_1$ at some red point $r_0$, so $L_2$ is a red line whose only green point is $x$. Observe that every other line that passes through $x$ will be a green line, as it must intersect $L$ at a point other than $z$. This implies that every red point in $PG(2,q)$ lies on $L_1$ or $L_2$ or is a single red point on the green line $\cl_P(\{p,x\})$. Therefore, for distinct red points $r_1,r_2 \text{ in } L_1 - \{r_0\}$, one of $\cl_P(\{r_1,z\})$ or $\cl_P(\{r_2,z\})$ will have at least two red points and two green points, a contradiction. Therefore, there cannot be a red point on any green line. By symmetry, there cannot be a green point on any red line. Since every two lines meet, this is a contradiction.\end{proof}


\section{Affine Target Matroids}\label{affinetargets}

In this section, we look at targets arising from affine geometries. This section begins with preliminary results about affine targets and minimal affine-non-targets. It concludes with the forbidden induced restrictions for affine targets over $GF(q)$. One fact that we use repeatedly is that if $(G,R)$ is a $2$-coloring of $AG(r-1,q)$, then $G$ is an affine target if and only if $R$ is an affine target. Viewing $AG(r-1,q)$ as a restriction, $PG(r-1,q)|X$, of $PG(r-1,q)$ obtained by deleting a projective hyperplane $H$ from $PG(r-1,q)$, we call $H$ the \textit{complementary hyperplane of} $X$. We shall also refer to $H$ as the \textit{complementary hyperplane of} $AG(r-1,q)$.

\begin{proposition}
Let $(E_0,E_1,\dots, E_k)$ and $(F_0,F_1, \dots, F_k)$ be nested sequences of flats in $AG(r-1,q)$ such that $r(E_i) = r(F_i)$ for all $i$ in $\{0,1,\dots k\}$. Let $H$ and $H'$ be the complementary hyperplanes of $E_k$ and $F_k$, respectively. Let $G_E$ and $G_F$ be the union, respectively, of all $E_{i+1}-E_i$ and of all $F_{i+1}-F_i$ for the even numbers $i$ in $\{0,1,\dots,k\}$. Then $AG(r-1,q)|G_E \cong AG(r-1,q)|G_F$.
\end{proposition}
\begin{proof}

        Observe that $E_k = E(PG(r-1,q)) - H$ and $F_k = E(PG(r-1,q)) - H'$. Let $h$ be the smallest $i$ such that $r(E_i) > 0$. Let $\{b_{h,1},b_{h,2},\dots,b_{h,m_h}\}$ and $\{d_{h,1},d_{h,2},\dots,d_{h,m_h}\}$ be bases $B_h$ and $D_h$ of $PG(r-1,q)|(\cl_P(E_h) - E_h)$ and $PG(r-1,q)|(\cl_P(F_h) - F_h)$, respectively. Let $v$ and $v'$ be elements in $E_h$ and $F_h$, respectively. Then $\{v, b_{h,1},b_{h,2},\dots,b_{h,m_h}\}$ is a basis for $PG(r-1,q)|\cl_P(E_h)$ and $\{v',d_{h,1},d_{h,2},\dots,d_{h,m_h}\}$ is a basis for $PG(r-1,q)|\cl_P(F_h)$.  Let $B_0 = B_1 = \dots = B_{h-1} = \emptyset$ and $D_0 = D_1 = \dots = D_{h-1} = \emptyset$. For $j \geq h$, assume that $B_0, B_1,\dots, B_j$ and $D_0,D_1,\dots,D_j$ have been defined. Let $B_{j+1}$ and $D_{j+1}$ be bases of $PG(r-1,q)|(\cl_P(E_{j+1}) - E_{j+1})$ and $PG(r-1,q)|(\cl_P(F_{j+1}) - F_{j+1})$, respectively, such that $B_j \subseteq B_{j+1}$ and $D_j \subseteq D_{j+1}$. Observe that adding $v$ and $v'$ to $B_{j+1}$ and $D_{j+1}$, respectively, gives bases for $PG(r-1,q)|\cl_P(E_{j+1})$ and $PG(r-1,q)|\cl_P(F_{j+1})$ for all $j$. Let $B_{j+1} - B_j = \{b_{j+1,1},b_{j+1,2},\dots,b_{j+1,m_{j+1}}\}$ and $D_{j+1} - D_j = \{d_{j+1,1},d_{j+1,2},\dots,d_{j+1,m_{j+1}}\}$. Observe that $B_k$ and $D_k$ are bases for $H$ and $H'$, respectively. Now, $G_E = \cl_P(G_E) - H$ and $G_F = \cl_P(G_F) - H'$. Define the automorphism $\phi$ on $PG(r-1,q)$ by $\phi(v) = v'$ and $\phi(b_{s,t}) = d_{s,t}$, for all $s$ and $t$ such that $s \geq h$. Then $\phi(H) = H'$ and, for all $i$, we have $\phi(\cl_P(B_i)) = \cl_P(D_i)$, so $\phi(\cl_P(B_{i+1}) - \cl_P(B_i) - H) = \phi(\cl_P(B_{i+1})) - \phi(\cl_P(B_i)) - \phi(H) = \cl_P(D_{i+1}) - \cl_P(D_i) - H'$. Thus, $PG(r-1,q)|(\cl_P(G_E) - H) \cong PG(r-1,q)|(\cl_P(G_F) -H')$. Therefore, $AG(r-1,q)|G_E \cong AG(r-1,q)|G_F$.
\end{proof}

Similar to projective targets, the previous result means that we can refer to a simple $GF(q)$-representable affine matroid $M$ as being an affine target when all the $GF(q)$-representations of $M$ are affine targets.

\begin{proposition}
The class of affine targets is closed under induced restrictions.
\end{proposition}
\begin{proof}
Let $(G,R)$ be a $2$-coloring of $AG(r-1,q)$. Assume that $G$ is an affine target. Then $G$ corresponds to a nested sequence $(F_0,F_1,\dots,F_k)$ of affine flats with $G$ being the union of the sets $F_{i+1}-F_i$ for all even $i$. Take a proper flat $X$ of $AG(r-1,q)$. As the intersection of two affine flats is an affine flat, the sequence $(X\cap F_0,X\cap F_1,\dots,X\cap F_k)$ is a nested sequence of affine flats. Assume that $n$ is odd. As $F_n - F_{n-1}\subseteq G$, it follows that $(X \cap F_n ) - (X \cap F_{n-1}) \subseteq G \cap F$. Hence, $G\cap F$ is the union of the sets $(X\cap F_{i+1}) - (X\cap F_i)$ for all even $i$. Therefore, $AG(r-1,q)|(G\cap X)$ is an affine target. 
\end{proof}

We will use the following well-known lemmas about affine geometries quite often in this section (see, for example [\ref{James}, Exercise 6.2.2]).
\begin{lemma} \label{affinepartition}
$AG(r-1,q)$ can be partitioned into $q$ hyperplanes. 
\end{lemma}
\begin{lemma}\label{affinehyperplaneint}
    Let $X$ and $Y$ be distinct hyperplanes of $AG(r-1,q)$. Then either $r(X\cap Y) = 0$, or $r(X\cap Y) = r - 2$.
\end{lemma}

The techniques used for handling affine targets are similar to those that we used for projective targets. The binary case will be treated separately.

\begin{lemma}\label{new}
    Let $(G,R)$ be a $2$-coloring of $AG(r-1,2)$ with $|G| = |R|$. Then $r(G) = r(R)$.
\end{lemma}

\begin{proof}
 Since $|G| = |R|$, we have that $|G| = 2^{r-2}$. Because the hyperplanes of $AG(r-1,2)$ have exactly $2^{r-2}$ elements, either $AG(r-1,2)|G$ is a hyperplane, or $r(G) = r$. Since $AG(r-1,2)|G$ is a hyperplane if and only if $AG(r-1,2)|R$ is a hyperplane, the lemma~follows.
\end{proof}

\begin{lemma}\label{BinaryAffineRank}
Let $(G,R)$ be a $2$-coloring of $AG(r-1,2)$. Assume $G$ is an affine target and $F$ is a flat of $AG(r-1,2)$. Then either exactly one of $G\cap F$ and $R\cap F$ is of rank $r(F)$; or $r(G\cap F) = r(R \cap F) = r(F)-1$, and each of $G \cap F$ and $R\cap F$ is an affine flat. Moreover, if $r(G \cap F) = r(F)$ and $H_1$ and $H_2$ are disjoint hyperplanes of $AG(r-1,2)|F$, then $r(G \cap H_1) = r(F) - 1$ or $r(G \cap H_2) = r(F) - 1$.
\end{lemma}

\begin{proof}
Assume $r(G\cap F) < r(F)$. Then there is a rank-$(r(F)-1)$ affine flat $H_G$ that is contained in $F$ and contains $G$. As $H_G$ is a hyperplane of $AG(r-1,2)|F$, there is another hyperplane $H_R$ of $AG(r-1,2)|F$ that is complementary to $H_G$ in $F$. Moreover, $H_R \subseteq R \cap F$, so $r(R \cap F) \geq r(F) - 1$. If there is a red point $z \text{ in }H_G$, then $r(R\cap F) = r(F)$. Otherwise, $r(R\cap F) = r(G\cap F) = r(F) - 1$, and each of $R\cap F$ and $G\cap F$ is an affine flat.

Now suppose that $r(G \cap F) = r(F)$ and that $H_1$ and $H_2$ are disjoint hyperplanes of $AG(r-1,2)|F$ with $r(G \cap H_1) < r(F) - 1$ and $r(G \cap H_2) < r(F) - 1$. As $AG(r-1,2)|(G \cap F)$ is an affine target of rank $r(F)$, there is a hyperplane $H'$ of $AG(r-1,2)|F$ that is monochromatic green. Since $H'$ must meet both of $H_1$ and $H_2$, its intersection with each such set has rank $r(F) - 3$. Since $F$ is green, it follows that $H_1$ or $H_2$ is green.
\end{proof}

\begin{lemma} \label{affinefullrank} 
Let $(G,R)$ be a $2$-coloring of $AG(r-1,q)$, where $q\geq 3$. Assume that $G$ is an affine target and $F$ is a flat of $AG(r-1,q)$. Then exactly one of $G \cap F$ and $R \cap F$ has rank $r(F)$.
\end{lemma}

\begin{proof}
Assume $r(G\cap F) < r(F)$. Then there is a rank-$(r(F)-1)$ affine flat $H_G$ containing $G \cap F$. Thus $F-H_G$ does not contain any green points, so $r(R\cap F) = r(F)$.
\end{proof}
As with $2$-colorings of $E(PG(r-1,q))$, for a $2$-coloring $(G,R)$ of $E(AG(r-1,q))$, a flat $F$ is \textit{green} if $r(G \cap F) = r(F)$. We call $F$ \textit{red} if $r(R\cap F) = r(F)$. Furthermore, a flat $F$ of $AG(r-1,2)$ is \textit{half-green and half-red} if $r(G\cap F) = r(R\cap F) = r(F) - 1$. In this case,  $G \cap F$ and $R \cap F$ are complementary hyperplanes of $AG(r-1,2)|F $. 

The following results show how one can get an affine target from a projective target and how to construct projective targets from affine targets. 

\begin{proposition}
    
   Let $(G,R)$ be a $2$-coloring of $PG(r-1,q)$. Let $H$ be a hyperplane of $PG(r-1,q)$. Assume that $G$ is a projective target. Then $PG(r-1,q)|(G - H)$ is an affine target.
\end{proposition}

\begin{proof}
    As $G$ is a projective target, $G$ corresponds to a nested sequence $(F_0,F_1,\dots,F_k)$ of projective flats, where $G$ is equal to the union of $F_{i+1} - F_i$ for all even $i$. Then $F_j - H$ is an affine flat for all $j$. Therefore, $(F_0 - H, F_1 - H,\dots, F_k - H)$ is a nested sequence of affine flats. Let $F_j' = F_j - H$ for all $j$. Then $PG(r-1,q)|(G - H)$ corresponds to the nested sequence $(F'_0,F'_1,\dots,F'_k)$ of affine flats and $G - H$  is equal to the union of $F'_{i+ 1} - F'_i$ for all even $i$.
\end{proof}
The following result is immediate.
\begin{proposition}\label{standardtarget}
    Let $(G,R)$ be a $2$-coloring of $AG(r-1,q)$. Assume that $G$ is an affine target corresponding to a nested sequence $(F_0,F_1,\dots,F_k)$ of affine flats where $G$ is equal to the union of $F_{i+1} - F_i$ for all even $i$. Viewing $AG(r-1,q)$ as a restriction of $PG(r-1,q)$, the sequence $(\cl_P(F_0),\cl_P(F_1),\dots,\cl_P(F_k))$ is a nested sequence of projective flats and, if $G_P$ is the projective target that is the union of $\cl_P(F_{i+1}) - \cl_P(F_i)$ for all even $i$, and $H = E(PG(r-1,q)) - E(AG(r-1,q))$, then $PG(r-1,q)|(G_P - H) \cong AG(r-1,q)|G$.
\end{proposition}

We call the projective target $G_P$ that arises from the affine target $G$ in Proposition \ref{standardtarget} the \textit{standard projective target arising from $G$}. Now consider an affine target $M_1$ that arises from a green-red coloring of $PG(r-1,q)\backslash H$ where $H$ is a projective hyperplane. Let $M_2$ be a projective target that arises as a green-red coloring of $H$. We say that $M_1$ and $M_2$ are \textit{compatible} if the green-red coloring of $PG(r-1,q)$ induced by the colorings of $M_1$ and $M_2$ is a projective target, that is, if $PG(r-1,q)|(E(M_1) \cup E(M_2))$ is a projective target. In the previous proposition, the affine target $G$ and the projective target $G_P \cap H$ are compatible as $PG(r-1,q)|(G \cup (G_P \cap H))$ is the projective target $G_P$. We now consider when $PG(r-1,q)|(E(M_1) \cup E(M_2))$ is not a standard projective target. As $M_1$ is an affine target, it corresponds to a canonical nested sequence $(F_0,F_1,\dots, F_k)$ of affine flats. Let $F_h$ be the first non-empty flat in this sequence. Then $\cl_P(F_h)$ meets the projective hyperplane $H$ in a rank-$(r(F_h)-1)$ projective flat $T$. In the construction of a standard projective target, $T$ is monochromatic. The next result shows that, apart from the standard projective target, the only way for $M_1$ and $M_2$ to be compatible is if we modify the standard projective target by replacing $T$ with a $2$-coloring of it that is a projective target.

\begin{proposition}
Let $(G,R)$ be a $2$-coloring of $PG(r-1,q)$. Let $H$ be a projective hyperplane. Assume that $PG(r-1,q)|(G - H)$ is an affine target corresponding to a canonical nested sequence $(F_0,F_1,\dots,F_k)$ of affine flats. Assume that $PG(r-1,q)|(G \cap H)$ is a projective target corresponding to a canonical nested sequence $(S_0,S_1,\dots,S_t)$ of projective flats. Then $PG(r-1,q)|(G-H)$ and $PG(r-1,q)|(G \cap H)$ are compatible if and only if, when $\beta$ is the smallest $h$ such that $r(F_h) > 0$,
\begin{enumerate}[label=(\roman*)]
    \item \label{one} there is an $m$ in $\{0,1,\dots,t\}$ such that $F_\beta \cup S_m$ is a projective flat, $r(S_m) = r(F_\beta) - 1,$ and $PG(r-1,q)|(G \cap (F_\beta \cup S_m))$ is a projective target; and

    \item \label{two} for all $\alpha$ in $\{1,2,\dots, k - \beta\}$, the set $F_{\beta + \alpha} \cup S_{m + \alpha}$ is a projective flat, $(F_{\beta + \alpha} \cup S_{m + \alpha}) - (F_{\beta + \alpha - 1} \cup S_{\beta + \alpha - 1})$ is monochromatic, and $t = m+k - \beta$.
\end{enumerate}
\end{proposition}
\begin{proof}
    Assume that $PG(r-1,q)|(G-H)$ and $PG(r-1,q)|(G\cap H)$ are compatible. Then $PG(r-1,q)|G$ is a projective target corresponding to a canonical nested sequence $(X_0,X_1,\dots,X_s)$ of projective flats. Thus $(X_0 \cap H, X_1 \cap H,\dots, X_s \cap H)$ is a nested sequence of projective flats for $PG(r-1,q)|H$, and $(X_0 - H, X_1 - H,\dots, X_s - H)$ is a nested sequence of affine flats for $PG(r-1,q)\backslash H$. Now,
    \begin{enumerate}[label=(\alph*)]
        \item\label{a} $X_1  = \emptyset$ and $X_2 \cap H = \emptyset$ but $X_3 \cap H \not= \emptyset$; or
        \item\label{b} $X_1 = \emptyset$ and $X_2 \cap H \not= \emptyset$; or
        \item\label{c} $X_1 \not= \emptyset$ but $X_1 \cap H = \emptyset$ and $X_2 \cap H \not=\emptyset$; or
        \item\label{d} $X_1 \not= \emptyset$ and $X_1 \cap H \not= \emptyset$.
    \end{enumerate}
For the projective target $PG(r-1,q)|H$, the canonical nested sequence is $(X_2 \cap H, X_3 \cap H, \dots, X_s \cap H)$ in case \ref{a} and is $(X_0 \cap H, X_1 \cap H, \dots, X_s \cap H)$ in the other three cases.

Let $\gamma$ be the smallest $h$ such that $X_h - H$ is non-empty. Then $\cl_P(X_\gamma - H)$ meets $H$ in a projective flat of rank $r(X_\gamma) - 1$. Thus $PG(r-1,q)|(G \cap X_\gamma \cap H)$ is a projective target in $X_\gamma \cap H$ that corresponds to the canonical nested sequence $(X_2 \cap H, X_3 \cap H, \dots, X_\gamma \cap H)$ in case \ref{a} and to the canonical nested sequence $(X_0\cap H, X_1 \cap H, \dots, X_\gamma \cap H)$ in the other three cases.

Now $(X_\gamma - X_{\gamma - 1}) - H = (X_\gamma - H) - (X_{\gamma - 1} - H) = (X_\gamma - H) - \emptyset$. Thus $X_\gamma - H$ is monochromatic. Therefore, the canonical nested sequence corresponding to $PG(r-1,q)|(G-H)$ is $(X_{\gamma - 1} - H, X_\gamma - H, \dots, X_s - H)$ when $X_\gamma - H$ is green and is $(\emptyset, X_{\gamma-1} - H, X_\gamma - H, \dots, X_s - H)$ when $X_\gamma - H$ is red. Thus $(F_0,F_1,\dots,F_k)$ is $(X_{\gamma - 1} - H, X_\gamma - H,\dots, X_s - H)$ when $X_\gamma - H$ is green and is $(\emptyset, X_{\gamma - 1} - H, X_\gamma - H,\dots, X_s - H)$ when $X_\gamma - H$ is red. We see that $F_\beta = X_\gamma - H$, that $F_\beta \cup (X_\gamma \cap H)$ is a projective flat, that $r(X_\gamma \cap H) = r(F_\beta) - 1$, and that $PG(r-1,q)|(G \cap (F_\beta \cup (X_\gamma \cap H))) = PG(r-1,q)|(G \cap X_\gamma)$. Therefore, $PG(r-1,q)|(G \cap (F_\beta \cup (X_\gamma \cap H)))$ is a projective target. Thus \ref{one} holds. Evidently $F_{\beta + \alpha} \cup (X_{\gamma + \alpha} \cap H) = X_{\gamma + \alpha}$, so $F_{\beta + \alpha} \cup (X_{\gamma + \alpha} \cap H)$ is a projective flat for all $\alpha$ in $\{1,2,\dots,k - \beta\}$. Moreover, $\gamma + k - \beta = s$ and \ref{two} holds.

Now suppose that \ref{one} and \ref{two} hold. We know that $S_m - S_{m-1}$ and $F_\beta$ are monochromatic. Then $PG(r-1,q)|G$ is a projective target for which the corresponding nested sequence is $(S_0,S_1,\dots,S_{m-1}, F_\beta \cup S_m, F_{\beta + 1} \cup S_{m+1},\dots, \linebreak F_k \cup S_t)$ when the colors of $S_m - S_{m-1}$ and $F_\beta$ match and is $(S_0,S_1,\dots,\linebreak S_{m}, F_\beta \cup S_m, F_{\beta + 1} \cup S_{m + 1}, \dots, F_k \cup S_t)$ when the colors of $S_m - S_{m-1}$ and $F_\beta$ differ. We conclude that $PG(r-1,q)|(G \cap H)$ and $PG(r-1,q)|(G-H)$ are compatible.
\end{proof}

A \textit{minimal affine-non-target} is an affine matroid that is not an affine target such that every proper induced restriction of it is an affine target. The next result is an analog of Lemma~\ref{rankcomp}.
\begin{lemma}\label{affinenontarget}
    Let $(G,R)$ be a $2$-coloring of $AG(r-1,q)$. Assume $G$ is a rank-r minimal affine-non-target. Then $r(R) = r$.
\end{lemma}

\begin{proof}
Assume $r(R) < r$. Then $R$ is contained in an affine hyperplane $H$. As $G$ is a minimal affine-non-target, $AG(r-1,q)|(G\cap H)$ is an affine target corresponding to a nested sequence $(F_0,F_1,\dots,F_{n-1},H)$ of affine flats. As $R\subseteq H$, there are no red points in $E(AG(r-1,q)) - H$. Then we obtain the contradiction that $G$ is an affine target for which a corresponding sequence of nested affine flats is $(F_0,F_1,\dots, F_{n-1}, H, E(AG(r-1,q)))$ if $H - F_{n-1} \subseteq R$ and $(F_0,F_1,\dots,F_{n-1},E(AG(r-1,q)))$ if $H-F_{n-1} \subseteq G$.
\end{proof}

\begin{lemma} \label{BinaryDisjointHP}
Let $(G,R)$ be a $2$-coloring of $AG(r-1,2)$. Assume $G$ is a minimal affine-non-target of rank $r$. Then $AG(r-1,2)$ has a red hyperplane and a green hyperplane that are disjoint.
\end{lemma}

\begin{proof}
Assume the lemma fails. By Lemma~\ref{affinenontarget}, $r(R) = r$, so we have a red hyperplane $X_1$ and a green hyperplane $Y_1$. There are affine hyperplanes $X_2$ and $Y_2$ that are complementary to $X_1$ and $Y_1$, respectively. As the lemma fails, $X_2$ is not green and $Y_2$ is not red. By assumption, $X_1$ and $Y_1$ meet in a rank-$(r-2)$ flat $F_{1,1}$. For $(i,j) \not= (1,1)$, let $F_{i,j} = X_i\cap Y_j$. As $r(F_{1,1}) = r-2$, it follows that $r(F_{i,j}) = r-2$ for each $i$ and $j$. Then $\{F_{1,1}, F_{1,2}, F_{2,1}, F_{2,2}\}$ is a partition of $AG(r-1,2)$ and there are red points in each of $F_{1,2}$ and $F_{1,1}$, and there are green points in each of $F_{1,1}$ and $F_{2,1}$. Next we show the following.

\begin{sublemma}\label{noredF21}
    There is a red point in $F_{2,1}$.
\end{sublemma}
As $AG(r-1,2)|(R\cap X_2)$ is a target and $r(G \cap X_2) < r-1$, it follows, by Lemma~\ref{BinaryAffineRank}, that either $r(R\cap X_2) = r-1$, or $R\cap X_2$ and $G\cap X_2$ are affine flats of rank $r-2$. In the first case, there is certainly a red point in $F_{2,1}$. Consider the second case. Assume that $F_{2,1}$ is monochromatic green. Then $F_{2,2}$ is monochromatic red. As $r(R\cap F_{1,2}) > 0$, we see that $r(R\cap Y_2) =  r - 1$, so $Y_2$ is red, a contradiction. Thus \ref{noredF21} holds.
\begin{sublemma}
    $F_{1,2}$ is red. \label{F21RedFlat}
\end{sublemma}
Assume that $F_{1,2}$ is not red. Then $r(R \cap F_{1,2}) < r-2$. As $X_1$ is red and $r(G \cap F_{1,1}) > 0$, it follows that $r(G \cap F_{1,2}) < r - 2$. Thus, by Lemma~\ref{BinaryAffineRank}, $R \cap F_{1,2}$ and $G\cap F_{1,2}$ are affine flats of rank $r-3$. Observe that $F_{1,1}$ is not red, otherwise $r(R\cap Y_1) = r-1$, a contradiction. Moreover, $F_{1,1}$ is not a green flat, otherwise $X_1$ is a green hyperplane. Thus $R\cap F_{1,1}$ and $G\cap F_{1,1}$ are affine flats of rank $r-3$. Now, as $r(R\cap X_1) = r-1$ and $|R \cap X_1| = |G \cap X_1|$, it follows by Lemma~\ref{new} that $r(G \cap X_1) = r-1$, a contradiction to Lemma~\ref{BinaryAffineRank}. Therefore, \ref{F21RedFlat} holds.

As $Y_2$ is not red but $F_{1,2}$ is red, $F_{2,2}$ is monochromatic green. Since $r(G \cap F_{2,1}) > 0$, we obtain the contradiction that $X_2$ is green. 
\end{proof}

In each of the remaining results in this section, we shall consider disjoint sets $\textbf{X}$ and $\textbf{Y}$ of hyperplanes of $AG(r-1,q)$ where the members of $\textbf{X}$ and $\textbf{Y}$ partition $E(AG(r-1,q))$. With $\textbf{X} = \{X_1,X_2,\dots,X_q\}$ and $\textbf{Y} = \{Y_1,Y_2,\dots,Y_q\}$, we let $F_{i,j} = X_i\cap Y_j$ for all $i$ and $j$. 

\begin{lemma}\label{qaryDisjointHP}
    Let $(G,R)$ be a $2$-coloring of $AG(r-1,q)$, where $q\geq 3$. Assume that $G$ is a minimal affine-non-target of rank $r$. Then $AG(r-1,q)$ has a red hyperplane and a green hyperplane that are disjoint.
\end{lemma}

\begin{proof}
    Assume the lemma fails. By Lemma~\ref{affinefullrank}, each proper flat of $AG(r-1,q)$ is either red or green but not both. By Lemma~\ref{affinenontarget}, $r(G) = r(R) = r$, so $AG(r-1,q)$ has a red hyperplane $X_1$ and a green hyperplane $Y_1$. Then there are partitions $\{X_1,X_2,\dots, X_q\}$ and $\{Y_1, Y_2,\dots, Y_q\}$ of  $E(AG(r-1,q))$ into sets \textbf{X} and \textbf{Y} of hyperplanes. By assumption, all the hyperplanes in  \textbf{X} are red and all the hyperplanes in  \textbf{Y} are green. As $X_1 \cap Y_1 \not= \emptyset$, it follows, by Lemma~\ref{affinehyperplaneint}, that $r(F_{i,j}) =r-2$ for all $i$ and $j$.

    As $X_1$ is red, at most one of $F_{1,1}, F_{1,2},\dots, F_{1,q-1}, \text{ and } F_{1,q}$ is green. Thus, we may assume that $F_{1,1}, F_{1,2},\dots, F_{1,q-2},$ and $ F_{1,q-1}$ are red. As $F_{1,1}$ is red and $Y_1$ is green, $Y_1- F_{1,1}$ will be monochromatic green. Similarly, $Y_2 - F_{1,2}$ will be monochromatic green. This implies that $r(G\cap X_2) = r(R\cap X_2) = r-1$, a contradiction. 
\end{proof}

The following technical lemmas show a relationship between the lines and planes of $AG(r-1,q)$ and the hyperplanes in \textbf{X} and \textbf{Y}. In these lemmas, when we take closures, we are doing so in the underlying affine geometry $AG(r-1,q)$.

\begin{lemma}\label{lineintersect}
Let $\textbf{X}$ and $\textbf{Y}$ be two disjoint sets each consisting of a set of hyperplanes that partition $AG(r-1,q)$. Let $x$ and $y$ be distinct elements of $E(AG(r-1,q))$ such that $|\{x,y\} \cap F_{i,j}| \leq 1$ for all $i,j$ and no member of \textbf{X} or \textbf{Y} contains $\{x,y\}$. Then $|\cl(\{x,y\})\cap X_i | = 1$ and $|\cl(\{x,y\} \cap Y_j| = 1$ for all $i$ and $j$.
\end{lemma}

\begin{proof}
    Clearly $|\cl(\{x,y\}) \cap X_i| \leq 1$ for all $i$, otherwise $X_i$ contains $\{x,y\}$. As $|\cl(\{x,y\})| = q$, we deduce that $|\cl(\{x,y\}) \cap X_i| = 1$ for all $i$. The lemma~follows by symmetry.
\end{proof}

\begin{lemma}\label{planeintersect}
For $q$ in $\{2,3\}$, let $\textbf{X}$ and $\textbf{Y}$ be two disjoint sets each consisting of a set of hyperplanes that partition $AG(r-1,q)$. Let $\{x,y,z\}$ be a rank-$3$ subset of $E(AG(r-1,q))$ such that $|\{x,y,z\} \cap F_{i,j}| \leq 1$ for all $i$ and $j$, and there is an $X_k$ in \textbf{X} such that $|X_k \cap \{x,y,z\}| = 2$. Then $ |\cl(\{x,y,z\}) \cap F_{i,j} | = 1$ for all $i$ and $j$.
\end{lemma}
\begin{proof}

Note that, by Lemma~\ref{affinepartition}, each $F_{i,j}$ has rank $r-2$. Let $q = 2$. We may assume that $x\in F_{1,1}$, that $y\in F_{1,2}$, and that $z \in F_{2,1}$. As $r(\{x,y,z\}) = 3$, there is exactly one point, say $e$, in $\cl(\{x,y,z\}) - \{x,y,z\}$. Assume $e\not\in F_{2,2}$. Then, by symmetry, we may assume that $e\in X_1$. As $\{e,x,y,z\}$ is a circuit, we deduce that $z\in X_1$, a contradiction.

Now assume that $q = 3$. We may assume that $x \in F_{1,1}$ and $y\in F_{1,2}$. Suppose $z\in F_{2,3}$. Consider $\cl(\{x,z\})$. The third point $e$ on this line cannot be in $X_1$, otherwise the circuit $\{e,x,z\}$ gives the contradiction that $z$ is in $X_1$. Similarly, $e$ cannot be in $X_2$, $Y_1$, or $Y_3$. Therefore, $e \in F_{3,2}$. By a similar argument, the third point on $\cl(\{y,z\})$ is in $F_{3,1}$. Continuing in this manner, we deduce that $|\cl(\{x,y,z\}) \cap X_3|= 3$. Since $|\cl(\{x,y,z\})| = 9$, using the same technique, we deduce that $|\cl(\{x,y,z\}) \cap F_{i,j}| = 1$ for all $i$ and $j$. By symmetry, we may now assume that $z \in F_{2,1}$. Then the third elements on the lines $\cl(\{x,z\}), \cl(\{x,y\}),$ and $\cl(\{y,z\})$ are in $F_{3,1}$, $F_{1,3}$, and $F_{3,3}$, respectively. Arguing as before, we again deduce that $|\cl(\{x,y,z\}) \cap F_{i,j}| = 1$ for all $i$ and $j$.
\end{proof}

\begin{lemma}\label{affineq2intersect2}
Let \textbf{X} and \textbf{Y} be two disjoint sets of each consisting of hyperplanes that partition $AG(r-1,2)$. Let $P_1 = \{w,x,y,z\}$ be a rank-$3$ flat of $AG(r-1,2)$ such that $w,x \in F_{1,2}$ and $y,z \in F_{2,1}$. Let $P_2 = \{e,f,y,z\}$ be a rank-$3$ flat of $AG(r-1,2)$ such that $e,f \in F_{2,2}$. Then $\cl(P_1 \cup P_2)$ is a rank-$4$ affine flat such that $|\cl(P_1 \cup P_2) \cap F_{i,j}| = 2$ for all $i$ and $j$. 
\end{lemma}
\begin{proof}
    As $|P_1 \cap P_2| = 2$, it follows that $r(P_1 \cup P_2) = 4$. Thus $\cl(P_1 \cup P_2)$ is a rank-$4$ affine flat. Now consider $\cl(\{e,w,z\})$. By Lemma~\ref{planeintersect}, $\cl(\{e,w,z\})$ intersects $F_{1,1}$ in an affine flat. Therefore, as $\cl(P_1 \cup P_2)$ meets each of $X_1, X_2, Y_1, Y_2, F_{1,1},F_{1,2},F_{2,1}$, and $F_{2,2}$ in an affine flat, each such intersection has $1$, $2$, or $4$ elements. Thus the lemma follows.
\end{proof}

We now prove the main results of this section. 

\begin{proof}[Proof of Theorem \ref{affinetargetq2}]
Assume $G$ is an affine target and there is a rank-$4$ affine flat $F$ such that $AG(r-1,2)|(G\cap F) \cong U_{4,4}$. Then $AG(r-1,2)|(R\cap F) \cong U_{4,4}$. This contradicts Lemma~\ref{BinaryAffineRank} as $r(G\cap F) = r(R\cap F) = 4$. Hence a binary affine target does not have $U_{4,4}$ as an induced restriction.

Let $(G,R)$ be a $2$-coloring of $AG(r-1,2)$. Suppose that $G$ is a rank-$r$ minimal affine-non-target and that $G$ does not contain $U_{4,4}$ as an induced restriction. By Lemma~\ref{BinaryDisjointHP}, $AG(r-1,2)$ has a red hyperplane $X_1$ and a green hyperplane $X_2$ such that $X_1\cap X_2 = \emptyset$. By Lemma~\ref{affinenontarget}, $r(R) = r$, so there is a red point $z \text{ in } X_2$. As $AG(r-1,2)|(R \cap X_1)$ is an affine target, in $X_1$, there is a monochromatic red rank-$(r-2)$ flat $F_{1,1}$. Observe that $\cl(F_{1,1} \cup z)$ is a red hyperplane $Y_1$ that intersects $X_1$ and $X_2$. Then there is a hyperplane $Y_2$ that is complementary to $Y_1$. By Lemma~\ref{affinehyperplaneint}, $r(F_{i,j}) = r-2$ for all $i$ and $j$. Observe that $z$ is in $F_{2,1}$. Furthermore, there is a red point $e \text{ in } F_{1,2}$ and there are green points $f$ and $g$ in $ F_{2,1}$ and $F_{2,2}$, respectively. As $r(G) = r$, there is a green point $h \text{ in }F_{1,2}$. We make the following observations.
\setcounter{theorem}{17} 
\begin{sublemma}\label{affinebinarydiag}
    $F_{1,2} \cup F_{2,1} $ is an affine hyperplane. 
\end{sublemma}
Observe that $F_{2,1}$ is contained in three affine hyperplanes, two of which are $F_{1,1} \cup F_{2,1}$ and $F_{2,1}\cup F_{2,2}$. Therefore, $F_{1,2}\cup F_{2,1}$ is the third such hyperplane. Thus \ref{affinebinarydiag} holds. 

\begin{sublemma}\label{notmonogreen}
     $F_{2,2}$ is not monochromatic green.
\end{sublemma}
Assume that $F_{2,2}$ is monochromatic green. Then $Y_2$ is green. By \ref{affinebinarydiag}, $F_{1,2} \cup F_{2,1}$ is an affine hyperplane, so both $AG(r-1,2)|(G\cap (F_{1,2} \cup F_{2,1}))$ and $AG(r-1,2)|(R\cap (F_{1,2} \cup F_{2,1}))$  are affine targets. Then there is a rank-$(r-2)$ affine flat $F$ such that either $G \cap (F_{1,2} \cup F_{2,1}) \subseteq F$ or $R \cap (F_{1,2} \cup F_{2,1}) \subseteq F$. Because we currently have symmetry between the red and green subsets of $AG(r-1,2)$, we may assume the former. Then $f$ and $h$ are in $F$. Let $x$ be a red point in $F_{1,2} - F$. Let $P_1 = \cl(\{f,h,x\})$. The fourth point $y$ on this plane is in $Y_1$, otherwise the circuit $\{f,h,x,y\}$ gives the contradiction that $f \in Y_2$. Moreover, $y\not\in F$, otherwise the circuit $\{f,h,x,y\}$ gives the contradiction that $x \in F$. Thus $y\in F_{1,2} - F$, so $y$ is red. Let $P_2 = \cl(\{f,g,y\})$. Then the fourth point $g'$ on this plane is in $F_{2,2}$, so $g'$ is green. By Lemma~\ref{affineq2intersect2}, $r(\cl(P_1 \cup P_2)) = 4$. Let $\{s,t\} = \cl(P_1 \cup P_2) - \{f,g,g',h,x,y\}$. Then, by Lemma~\ref{affineq2intersect2}, $s$ and $t$ in $F_{1,1}$, so both points are red. Therefore, $r(G \cap \cl(P_1 \cup P_2)) = r( R\cap \cl(P_1 \cup P_2)) = 4$, so $AG(r-1,2)|\{f,g,g',h\} \cong U_{4,4}$. We conclude that $AG(r-1,2)|G$ has $U_{4,4}$ as an induced restriction, a contradiction. Thus \ref{notmonogreen} holds.

The affine hyperplane $F_{1,2} \cup F_{2,1}$ is either green, red, or half-green and half-red. 
\begin{sublemma} \label{notred}
    $F_{1,2} \cup F_{2,1}$ is not red
\end{sublemma}
Assume that $F_{1,2} \cup F_{2,1}$ is red. Then, by Lemma~\ref{BinaryAffineRank}, at least one of $F_{1,2}$ and $F_{2,1}$ will be red. Assume that $F_{2,1}$ is red. As $X_2$ is green, $F_{2,2}$ is monochromatic green, otherwise $r(R \cap X_2) = r-1$. By \ref{notmonogreen}, we deduce that $F_{2,1}$ is not red. Thus $F_{1,2}$ is red. Observe that if $F_{2,1}$ is green, then $r(G \cap (F_{1,2} \cup F_{2,1})) = r(R \cap (F_{1,2} \cup F_{2,1})) = r-1$, a contradiction. Thus, $F_{2,1}$ is half-green and half-red. As $X_2$ is green, by Lemma~\ref{BinaryAffineRank}, $F_{2,2}$ is green. Therefore $r( G \cap (F_{2,2}\cup h)) = r-1$, so $Y_2$ is green. As $F_{1,2}$ is red, $F_{2,2}$ is monochromatic green, a contradiction to \ref{notmonogreen}. Therefore, \ref{notred} holds.

Since $F_{1,2} \cup F_{2,1}$ is not red, there is a monochromatic green flat $Z$ of rank $r-2$ that is contained in $F_{1,2} \cup F_{2,1}$. Because neither $F_{1,2}$ nor $F_{2,1}$ is monochromatic green, $Z$ meets $F_{1,2}$ and $F_{2,1}$ in monochromatic green flats, $Z_{1,2}$ and $Z_{2,1}$, of rank $r-3$. Similarly, as $X_2$ is green, there is a monochromatic green flat $V$ of rank $r-2$ that is contained in $X_2$. Because neither $F_{2,1}$ nor $F_{2,2}$ is monochromatic green, $V$ meets $F_{2,1}$ and $F_{2,2}$ in monochromatic green flats, $V_{2,1}$ and $V_{2,2}$, of rank $r-3$.

In the next part of the argument, we shall use the observation that if $Y_2$ is green, then we have symmetry between $(X_1,X_2)$ and $(Y_1,Y_2)$.

\begin{sublemma}\label{notgreen}
    $F_{1,2} \cup F_{2,1}$ is not green.
\end{sublemma}

Assume that $F_{1,2} \cup F_{2,1}$ is green. Then, by Lemma~\ref{BinaryAffineRank}, $F_{2,1}$ or $F_{1,2}$ is green. But the latter case implies that $Y_2$ is green, so this case can be reduced to the former by the symmetry between $(X_1,X_2)$ and $(Y_1,Y_2)$ noted above. Thus we may assume that $F_{2,1}$ is green.

Now $Z_{2,1}$ and $V_{2,1}$ are rank-$(r-3)$ monochromatic green flats that are both contained in $F_{2,1}$. Suppose $Z_{2,1} = V_{2,1}$. As $F_{2,1}$ is green, there is a  green element $g_1$ in $F_{2,1} - Z_{2,1}$. Since $F_{2,2}$ is not monochromatic green, there is a red point $u_1$ in $F_{2,2} - V_{2,2}$. Take a green point $g_2$ in $V_{2,2}$ and let $P_1 = \cl(\{g_1,g_2,u_1\})$. Let the fourth point on this plane be $g_3$. Then $g_3 \in F_{2,1}$, otherwise the circuit $\{g_1,g_2,g_3,u_1\}$ implies that $g_1 \in F_{2,2}$, a contradiction. Likewise, $g_3 \in V_{2,1}$, otherwise $g_2\not\in V$, a contradiction. Because $X_1$ is red, there is a red point $u_2$ in $F_{1,2} - Z_{1,2}$. Let $P_2 = \cl(\{g_1,g_3,u_2\})$. Let $g_4$ be the fourth point on this plane. Then the circuit $\{g_1,g_3,g_3,u_2\}$ implies that $g_4~\in~Z_{1,2}$, so $g_4$ is green. By Lemma~\ref{affineq2intersect2}, $\cl(P_1 \cup P_2)$ is a rank-$4$ affine flat having two points, $s$ and $t$, in $F_{1,1}$. We see that $AG(r-1,2)|\{g_1,g_2,g_3,g_4\} \cong U_{4,4}$. Thus $G$ has $U_{4,4}$ as an induced restriction, a contradiction. Thus $Z_{2,1} \not= V_{2,1}$.

Now $Y_2$ contains the monochromatic green flats $Z_{1,2}$ and $V_{2,2}$, each of which has rank $r-3$. Thus $Y_2$ is green, or $Y_2$ is half-green and half-red. Assume the latter. Then $Z_{1,2} \cup V_{2,2}$ is a monochromatic green flat of rank $r-2$ and $Y_2-(Z_{1,2} \cup V_{2,2})$ is a monochromatic red flat of rank $r-2$. As before, we take $u_1$ to be a red point in $F_{2,2}$. Choose $g_1$ to be a point in $V_{2,2}$. Then $g_1$ is green. Let $g_2$ be a point in $Z_{2,1} - V_{2,1}$, so $g_2$ is green. Let $P_1 = \cl(\{g_1,g_2,u_1\})$ and let $g_3$ be the fourth point in $P_1$. Then $g_3 \in F_{2,1}$ and $g_3 \in V$. Thus $g_3 \in V_{2,1}$, so $g_3$ is green. Choose $u_2$ in $F_{1,2} - Z_{1,2}$. Then $u_2$ is red. Let $P_2 = \cl(\{g_1,u_1,u_2\})$ and let $g_4$ be the fourth point in $P_2$. Then $g_4 \in F_{1,2}$ and $g_4 \in Z_{1,2} \cup V_{2,2}$, so $g_4 \in Z_{1,2}$. Thus $g_4$ is green. By Lemma \ref{affineq2intersect2}, $\cl(P_1 \cup P_2)$ is a rank-$4$ affine flat that meets $F_{1,1}$ in two elements, both of which are red. Moreover, $AG(r-1,2)|\{g_1,g_2,g_3,g_4\} \cong U_{4,4}$, a contradiction. 

We now know that $Y_2$ is green. Then there is a monochromatic green flat $W$ of rank $r-2$ such that $W \subseteq Y_2$. As neither $F_{1,2}$ nor $F_{2,2}$ is monochromatic green, $W \cap F_{1,2}$ and $W \cap F_{2,2}$ are monochromatic green flats, $W_{1,2}$ and $W_{2,2}$, of rank $r-3$. We choose $u_1$ to be a red point in $F_{2,2} - (V_{2,2} \cup W_{2,2})$. Choose $g_1$ in $V_{2,2} \cup W_{2,2}$. Then $g_1$ is green. Choose $g_2$ in $Z_{2,1} - V_{2,1}$. Then $g_2$ is green. The fourth point $g_3$ of the plane $P_1$ that equals $\cl(\{g_1,g_2,u_1\})$ is in $F_{2,1} \cap V$; that is, $g_3 \in V_{2,1}$, so $g_3$ is green. Now let $u_2$ be a red point in $F_{1,2} - (Z_{1,2} \cup W_{1,2})$. The fourth point $g_4$ on the plane $P_2$ that equals $\cl(\{g_1,u_1,u_2\})$ is in $F_{1,2} \cap W$, so it is in $W_{2,1}$ and hence is green. Then, by Lemma~\ref{affineq2intersect2}, $\cl(P_1\cup P_2)$ is a rank-$4$ affine flat that contains exactly four green points $g_1,g_2,g_3,$ and $g_4$. Since $AG(r-1,2)|\{g_1,g_2,g_3,g_4\} \cong U_{4,4}$, we have a contradiction. We conclude that \ref{notgreen} holds.

By \ref{notred} and \ref{notgreen}, we must have that $F_{1,2} \cup F_{2,1}$ is half-green and half-red. As $Z$ is a monochromatic green flat of rank $r-2$ that is contained in $F_{1,2} \cup F_{2,1}$, we deduce that $(F_{1,2} \cup F_{2,1}) - Z$ is a monochromatic red flat of rank $r-2$. Moreover, $F_{1,2}- Z$ and $F_{2,1}- Z$ are monochromatic red flats of rank $r-3$. Thus $V_{2,1} = Z_{2,1}$. As $X_2$ is green, there is a green point $g_1$ in $F_{2,2} - V$. Take $g_2$ to be a point in $V_{2,2}$ and let $u_1$ be a point in $F_{2,1} - V_{2,1}$. Let $P_1 = \cl(\{g_1,g_2,u_1\})$. The fourth point $g_3$ on this plane is in $F_{2,2}$ and in $V$ so it is in $V_{2,2}$ and hence it is green. Let $u_2$ be a point in $F_{1,2} - Z_{1,2}$. Then $u_2$ is red. Let $P_2 = \cl(\{g_3,u_1,u_2\})$. The fourth point $g_4$ on this plane is in $F_{1,2} \cap Z$, so it is green. By Lemma~\ref{affineq2intersect2}, $\cl(P_1 \cup P_2)$ is a rank-$4$ affine flat that contains exactly four green points, $g_1,g_2,g_3,$ and $g_4$. Moreover, $AG(r-1,2)|\{g_1,g_2,g_3,g_4\} \cong U_{4,4}$, a contradiction. We conclude that the theorem holds.
\end{proof}

\begin{proof}[Proof of Theorem \ref{affinetargetq3}]
Assume that $G$ is an affine target over $GF(3)$ such that there is an affine flat $F$ for which $AG(r-1,3)|(G\cap F)$ is one of $U_{3,3}, U_{3,4}, U_{2,3}\oplus U_{1,1}, U_{2,3} \oplus_2 U_{2,4}, P(U_{2,3},U_{2,3}),$ or $\mathcal{W}^3$. Then $r(G\cap F) = r(R\cap F) = 3$, contradicting Lemma~\ref{affinefullrank}.  
  
Let $(G,R)$ be a $2$-coloring of $AG(r-1,3)$. Suppose that $G$ is a rank-$r$ minimal affine-non-target. Then $r(G) \geq 3$. If $r(G) = 3$, then, by Lemma~\ref{affinenontarget}, $r(R) = 3$. One can now check that $AG(r-1,3)|G$ is one of $U_{3,3}, U_{3,4}, U_{2,3}\oplus U_{1,1}, U_{2,3} \oplus_2 U_{2,4}, P(U_{2,3},U_{2,3}),$ or $\mathcal{W}^3$. Thus we may assume $r(G) \geq 4$ and that $G$ does not contain a rank-$3$ flat $F$ such that $r(G\cap F) = r(R\cap F) = 3$. By Lemma~\ref{affinenontarget}, $r(R) = r$. Now, by Lemma~\ref{qaryDisjointHP}, there is a green hyperplane $X_1$ and a red 
hyperplane $X_2$ that are disjoint. Let $\{X_1,X_2,X_3\}$ and $\{Y_1,Y_2,Y_3\}$ be distinct sets, \textbf{X} and \textbf{Y}, each consisting of three disjoint hyperplanes in $AG(r-1,3)$. Then, by Lemma~\ref{affinehyperplaneint}, $r(F_{i,j}) = r -2$ for all $i$ and $j$. We proceed by showing there are no possible colorings of the hyperplanes in \textbf{Y}. 
\setcounter{theorem}{18}
\begin{sublemma}\label{affineq3case1}
 If $F_{1,1}$ and $F_{1,2}$ are green, then $Y_1$ or $Y_2$ is green.
\end{sublemma} 
Assume that $Y_1$ and $Y_2$ are both red. Then $Y_1 - F_{1,1}$ and $Y_2 - F_{1,2}$ are monochromatic red. As $r(G) = r$, there is a green element $e$ in $Y_3 - F_{1,3}$. Let $f$ and $g$ be green elements in $F_{1,1}$ and $F_{1,2}$, respectively. Consider $\cl(\{e,f,g\})$. By Lemma~\ref{planeintersect} this plane will contain red points in $F_{2,1}, F_{2,2},$ and $F_{3,2}$. Therefore $r(G\cap \cl(\{e,f,g\})) = r(R\cap \cl(\{e,f,g\})) = 3$, a contradiction. Thus \ref{affineq3case1} holds.
  
\begin{sublemma} \label{q32h}
    There cannot be at least two red hyperplanes or at least two green hyperplanes in $\mathbf{Y}$. 
\end{sublemma} 
Assume that $Y_1$ and $Y_2$ are red. As $X_1$ is green, at most one of $F_{1,1}, F_{1,2}$, and $F_{1,3}$ is red. By \ref{affineq3case1}, we may assume that $F_{1,2}$ is red. Then $X_1 - F_{1,2}$ is monochromatic green, so $Y_1 - F_{1,1}$ is monochromatic red. As $r(G) = r$, there is a green point $g$ that is not in $X_1$. Then $g \in F_{2,2} \cup F_{2,3} \cup F_{3,2} \cup F_{3,3}$. Let $e$ be a green point in $F_{1,1}$, and $f$ be a red point in $F_{1,2}$. Consider $\cl(\{e,f,g\})$. By Lemma~\ref{planeintersect}, this plane will contain red points in $F_{2,1}$ and $F_{3,1}$, and a green point in $F_{1,3}$. Therefore, $r(G\cap \cl(\{e,f,g\})) = r(R \cap \cl(\{e,f,g\})) = 3$, a contradiction. By symmetry, there cannot be two green hyperplanes in \textbf{Y}. Thus \ref{q32h} holds.

We conclude that there are no possible colorings of the hyperplanes in \textbf{Y}, a contradiction. 
\end{proof}

\begin{proof}[Proof of Theorem \ref{affinetargetsq4}]
Assume that $G$ is an affine target over $GF(q)$ for $q \geq 4$. If there is an affine flat $F$ such that $AG(r-1,q)|(G\cap F)$ is any of $U_{2,2},U_{2,3},\dots, U_{2,q-3},$ or $ U_{2,q-2}$, then $r(G\cap F) = r(R\cap F)$, contradicting Lemma~\ref{affinefullrank}.
  
Let $(G,R)$ be a $2$-coloring of $AG(r-1,q)$. Suppose that $G$ is a rank-$r$ minimal affine-non-target that does not contain $U_{2,2},U_{2,3},\dots, U_{2,q-3},$ or $ U_{2,q-2}$ as an induced restriction. Then $r(G) \geq 3$. By Lemma~\ref{affinefullrank}, $r(R) = r$. Now, by Lemma~\ref{qaryDisjointHP}, there is a red hyperplane $X_1$ that is disjoint from a green hyperplane $X_2$. Let $\{X_1, X_2,\dots, X_q\}$ and $\{Y_1,Y_2,\dots, Y_q\}$ be disjoint sets, $\textbf{X}$ and $\textbf{Y}$, each consisting of $q$ disjoint hyperplanes in $AG(r-1,q)$. By Lemma~\ref{affinehyperplaneint}, $r(F_{i,j}) = r - 2$ for all $i$ and $j$.  We show there are no
possible colorings of the hyperplanes in $\textbf{Y}$. 
\setcounter{theorem}{19}
\begin{sublemma}\label{affinesame} 
   There is at least one green hyperplane and at least one red hyperplane in \textbf{Y}. 
\end{sublemma} 
    Assume that all members of \textbf{Y} are red. As $X_2$ is green, we may assume that $F_{2,1},F_{2,2},\dots, F_{2,q-2}$, and $F_{2,q-1}$ are green. Then $Y_k - F_{2,k}$ is monochromatic red for all $k$ in $\{1,2,\dots,q-1\}$. As $r(G) = r$, there is a green element $e$ in $Y_q-F_{2,q}$. We may assume that $e \in F_{3,q}$. Let $f$ be a green element in $F_{2,1}$. Consider $\cl(\{e,f\})$. By Lemma~\ref{lineintersect}, this line will contain red points in $Y_2 - (F_{2,2} \cup F_{3,2})$ and $Y_3 - (F_{2,3} \cup F_{3,3})$. However, this gives the contradiction that $r(G\cap \cl(\{e,f\})) = r(R \cap \cl(\{e,f\})) = 2$. By symmetry, not all members of \textbf{Y} are green. Thus \ref{affinesame} holds.
    
\begin{sublemma} \label{2green2red}
There cannot be at least two green hyperplanes and at least two red hyperplanes in \textbf{Y}. 
\end{sublemma} 
  
Let $Y_1$ and $Y_2$ be green and let $Y_3$ and $Y_4$ be red. As $X_1$ is red, at most one of $F_{1,1},F_{1,2},\dots, F_{1,q-1}, \text { and } F_{1,q}$ is green. This implies that $F_{1,1}$ or $F_{1,2}$ is red, so we may assume the latter. Then $Y_2 - F_{1,2}$ is monochromatic green. Similarly, as $X_2$ is green,  $F_{2,3}$ or $F_{2,4}$, say $F_{2,3}$, is green. Then $Y_3 - F_{2,3}$ is monochromatic red. Assume $F_{1,1}$ and $F_{2,1}$ are green. Then $X_1 - F_{1,1}$ is monochromatic red. Let $e$ be a red point in $F_{1,4}$ and let $f$ be a green point in $F_{2,1}$. Consider $\cl(\{e,f\})$.  Then, by Lemma~\ref{lineintersect}, this line will have a green point in $Y_2 - (F_{1,2} \cup F_{2,2})$ and a red point in $Y_3 - (F_{1,3} \cup F_{2,3})$. Hence $r(G\cap \cl(\{e,f\})) = r(R\cap \cl(\{e,f\}))$, a contradiction. A symmetric argument holds when $F_{1,4}$ and $F_{2,4}$ are both red. Therefore, either $F_{1,1}$ or $F_{2,1}$ is red, and either $F_{1,4}$ or $F_{2,4}$ is green. This implies that $Y_1 - (F_{1,1} \cup F_{2,1})$ is monochromatic green and $Y_4 - (F_{1,4} \cup F_{2,4})$ is monochromatic red. Hence $r(G \cap X_3) = r(R \cap X_3) = r-1$, a contradiction. Thus \ref{2green2red} holds.

\begin{sublemma}\label{affineq41h}
    There cannot be exactly one red hyperplane or exactly one green hyperplane in \textbf{Y}. 
\end{sublemma} 
  
Assume that $Y_1$ is red and $Y_2,Y_3,\dots, Y_{q-1}$, and $Y_q$ are green. As $X_1$ is red, at most one of $F_{1,1},F_{1,2},\dots, F_{1,q-1}, \text {and } F_{1,q}$ is green. First assume that $F_{1,2},F_{1,3},\dots, F_{1,q-1}$ and $F_{1,q}$ are red. Then $Y_k - F_{1,k}$ is monochromatic green for all $k$ in $\{2,3,\dots,q\}$. By Lemma~\ref{affinefullrank}, $r(R) = r$, so there is a red point $e$ in $Y_1 - F_{1,1}$. We may assume $e$ is in $F_{2,1}$. Let $f$ be a red point in $F_{1,2}$ and consider $\cl(\{e,f\})$. By Lemma~\ref{lineintersect}, this line will have green points in $X_3 - (F_{3,1} \cup F_{3,2})$ and $X_4 - (F_{4,1} \cup F_{4,2})$. Therefore, $r(G \cap \cl(\{e,f\})) = r (R \cap \cl(\{e,f\})) = 2$, a contradiction. 
  
Now assume that $F_{1,2}$ is green. Then $X_1 - F_{1,2}$ is monochromatic red. Hence $Y_k - F_{1,k}$ is monochromatic green for all $k$ in $\{3,4,\dots,q\}$. Let $e$ be a red element in $Y_1 - F_{1,1}$ and $f$ be a green element in $Y_2 - F_{1,2}$ such that $|X_i \cap \{e,f\}| \leq 1$ for all $i$ in $\{2,3,\dots,q\}$. As $Y_1$ is red and $Y_2$ is green, such a pair of points exists. We may assume that $e \in F_{2,1}$ and $f \in F_{3,2}$. Then, by Lemma~\ref{lineintersect}, $\cl(\{e,f\})$ will contain a red element in $X_1 - (F_{1,1} \cup F_{1,2})$ and a green element in $X_4 - (F_{4,1} \cup F_{4,2})$, a contradiction. By symmetry, there cannot be exactly one green hyperplane in \textbf{Y}. Thus \ref{affineq41h} holds.

We conclude that there are no possible colorings of the hyperplanes in \textbf{Y}, a contradiction. 
\end{proof}

\end{document}